\documentclass[12pt]{amsart}
\usepackage[english]{babel}
\usepackage[T1]{fontenc}
\usepackage{microtype}
\usepackage{amsmath,amstext,amsfonts,amssymb,amsthm}
\usepackage{mathtools}
\usepackage{dirtytalk}
\usepackage{enumitem}
\usepackage{needspace}
\usepackage[colorlinks=true,allcolors=blue]{hyperref}
\usepackage{theoremref}
\allowdisplaybreaks

\newtheorem{theorem}{Theorem}[section]
\newtheorem{lemma}[theorem]{Lemma}
\newtheorem{prop}[theorem]{Proposition}

\theoremstyle{definition}

\theoremstyle{remark}

\newtheorem{thmx}{\textbf{Theorem}}

\newcommand{\cA}{\mathcal{A}}
\newcommand{\cC}{\mathcal{C}}

\newcommand{\cK}{\mathcal{K}}

\newcommand{\cU}{\mathcal{U}}

\newcommand{\bR}{\mathbb{R}}

\newcommand{\Sub}{\operatorname{Sub}}

\newcommand{\CC}{\mathcal{CC}}
\newcommand{\G}{\Gamma}

\DeclareMathOperator{\prob}{Prob}
\DeclareMathOperator{\Conv}{Conv}
\DeclareMathOperator{\Max}{Max}
\DeclareMathOperator{\Min}{Min}
\title[IRSs of \(C^*\)-simple groups]{Invariant random subgroups of \(C^*\)-simple groups}
\date{}
\author[Amrutam]{Tattwamasi Amrutam}
\address{Institute of Mathematics of the Polish Academy of Sciences, ul. Sniadeckich 8, 00-656, Warszawa, Poland}
\email{tattwamasiamrutam@gmail.com}
\thanks{T.A. is supported by National Science Centre, Poland Sonata, grant number 2025/59/D/ST1/03117.}
\author[Kalantar]{Mehrdad Kalantar}
\address{Mehrdad Kalantar\\
University of Oxford, UK}
\email{mehrdad.kalantar@maths.ox.ac.uk}
\begin{document}
\raggedbottom
\begin{abstract}
Let $\Gamma$ be a countable discrete $C^*$-simple group. Let $\mu$ be a
$\Gamma$-invariant Borel probability measure on $\text{Sub}(\Gamma)$,
i.e., an IRS. We show that $\mu$-almost every subgroup $H$ is
$C^*$-simple.
\end{abstract}
\maketitle
\section{Introduction}
A countable discrete group $\Gamma$ is called $C^*$-simple if its
reduced group $C^*$-algebra $C_r^*(\Gamma)$ is simple. It has the
\textit{unique trace property} (UTP) if the canonical trace is the only
tracial state on $C_r^*(\Gamma)$. Recall that the amenable
radical, denoted by $\operatorname{Rad}(\Gamma)$, is the largest
amenable normal subgroup of $\Gamma$. Breuillard, Kalantar, Kennedy
and Ozawa \cite[Theorem~1.3]{BKKO} proved that triviality of the amenable radical is equivalent to the unique trace property.

Since the amenable radical is a characteristic subgroup, if $\G$ has the UTP, then every normal subgroup of $\G$ also has the UTP.
The similar result for $C^*$-simplicity was proved in \cite{BKKO}, answering a question of de la Harpe and Pr{\'e}aux~\cite{de2011c}.

This paper concerns the extensions of these inheritance questions.

Denoting by ${\rm Sub}(\G)$ the Chabauty space of subgroups of $\G$, equipped with the continuous $\G$-action by conjugations,
a minimal component of ${\rm Sub}(\G)$ is called a \emph{uniformly recurrent subgroup} (URS), and a $\G$-invariant probability measure $\mu$ on ${\rm Sub}(\G)$ is called an \emph{invariant random subgroup} (IRS) on $\G$ \cite{AGV}.

The study of URS’s and IRS’s is a fast-developing topic in dynamics and ergodic theory. These are important generalizations of normal subgroups: while normal subgroups are kernels of group actions, these generalizations correspond to the point stabilizers of topological and measurable group actions.

It is now well-understood that (amenable) IRSs and URSs of $\G$ are closely connected to the UTP and $C^*$-simplicity of $\G$. More precisely, $\G$ has UTP if and only if it has no non-trivial amenable IRS \cite{BDL, BKKO}, and $\G$ is $C^*$-simple if and only if it has no non-trivial amenable URS \cite{Kennedy}.

It is therefore natural to ask whether the inheritance of
$C^*$-simplicity and the unique trace property extends from normal
subgroups to IRSs and URSs. For IRSs, the unique trace property does pass to almost every
subgroup. This follows from \cite[Corollary~1.5]{BDL} and \cite{BKKO}.

The corresponding assertion for URSs fails. Indeed, Le Boudec \cite{LB}
constructed groups with trivial amenable radical which are not
$C^*$-simple. Such a group has the unique trace property by
\cite[Theorem~1.3]{BKKO}, but it has a nontrivial amenable URS by \cite[Theorem~1.2]{Kennedy}. Every member of this URS is nontrivial
and amenable, and therefore fails the unique trace property.
Obviously, UTP does not imply any $C^*$-simplicity for URSs or IRSs.
Recently, Bray and Kennedy \cite[Corollary~3.3]{BK2026} (also see \cite{HK2026}) proved
that every URS of a countable $C^*$-simple group contains a
$C^*$-simple subgroup (in fact, they proved a generic member of the URS is $C^*$-simple).

In this paper, we complete this picture by proving the last remaining piece. We show that $C^*$-simplicity passes to almost every subgroup of
every IRS. This extends the normal subgroup result to invariant
random subgroups.
\begin{thmx}\label{thm:maintheorem}
Let $\Gamma$ be a countable discrete $C^*$-simple group and let $\mu$ be
an invariant random subgroup of $\Gamma$. Then $H$ is $C^*$-simple for
$\mu$-almost every $H\in\Sub(\Gamma)$.
\end{thmx}

Our proof also establishes a measurability result for Furstenberg URSs. For a
subgroup $H\leq\Gamma$, let $\cA_H$ denote the Furstenberg URS of
$H$, viewed as a compact family of amenable subgroups of $\Gamma$ (see~Section~\ref{Sec:preliminaries}).
Writing $\Sub^{\mathrm{am}}(\Gamma)$ for the space of amenable
subgroups of $\Gamma$ and $\cK(\Sub^{\mathrm{am}}(\Gamma))$ for
its compact subsets, equipped with the Vietoris topology, we show that the map $H\mapsto\cA_H$ from $\Sub(\Gamma)$ to
$\cK(\Sub^{\mathrm{am}}(\Gamma))$ is Borel measurable (see
Proposition~\ref{prop:AH-borel}).
\section{Preliminaries}\label{Sec:preliminaries}
In this section, we briefly recall the notions of Furstenberg boundaries, uniformly recurrent subgroups (URSs) and the associated Furstenberg URS. We also prove measurability of certain naturally associated maps. Throughout, $\Gamma$ is a countable discrete group.
\subsection{Invariant random subgroups}
We view $\Sub(\Gamma)$ as a closed subset of $\{0,1\}^{\Gamma}$ by
identifying a subgroup with its characteristic function. The resulting
topology is called the Chabauty topology. Equipped with this topology, $H_n\to H$ if and only if, for every $g\in\Gamma$,
$\chi_{H_n}(g)\to \chi_H(g)$. Moreover, equipped with this topology, $\Sub(\Gamma)$ is compact and
metrizable, and the conjugation action of $\Gamma$ on $\Sub(\Gamma)$ is
continuous.

A Borel probability measure $\mu$ on $\Sub(\Gamma)$ is called an
\textit{invariant random subgroup}, or an \textit{IRS}, if it is invariant
under the conjugation action of $\Gamma$. The notion of IRS was introduced by Ab\'ert, Glasner and Vir\'ag~\cite{AGV} and since then, it has opened up a great many important directions of research.

We write $\Sub^{\mathrm{am}}(\Gamma)$ for the set of amenable subgroups of
$\Gamma$. It is well known that this set is closed in $\Sub(\Gamma)$ (see, for example, \cite[Propositions~4.1 and~4.7]{AHO}).
\subsection{Topology on closed subsets of a compact space}
For a compact metrizable
space $Z$ we denote by $\cK(Z)$ the space of compact subsets of
$Z$, equipped with the topology generated by the sets $\{K\in \cK(Z):K\cap C=\varnothing\}$ and $\{K\in \cK(Z):K\cap V\ne\varnothing\}$ with $C$ varying over all closed subsets of $Z$, and $V$ varying over all open subsets of $Z$. The space $\cK(Z)$ is again a compact
metrizable space (see e.g. \cite[Section 4.F]{Kechris} for more details).

Since $Z$ is compact and metrizable, every open $V\subseteq Z$ can be written as $V=\cup_n C_n$, where $C_n$ is compact for all $n$, and therefore $\{K\in \cK(Z):K\cap V\ne\varnothing\} = \bigcup_n \{K\in \cK(Z):K\cap C_n=\varnothing\}^c$.
{In particular, a map from a measurable space into $\cK(Z)$ is measurable if and only if the inverse image of $\{K\in \cK(Z):K\cap C\ne\varnothing\}$ is measurable for every closed set $C\subseteq Z$.}

Any action of a group $\Gamma$ on $Z$ by homeomorphisms induces a canonical action of $\Gamma$ on $\cK(Z)$ by homeomorphisms. {In this case, we denote by $\cK^\G(Z)$ the collection of nonempty $\G$-invariant closed subsets of $Z$. This is a closed subset of $\cK(Z)$, since the fixed point set is closed and the empty set is an isolated point of $\cK(Z)$.}

\subsection{Furstenberg boundaries}
\label{subsec:boundaries}
Let $X$ be a compact Hausdorff space on which $\Gamma$ acts by
homeomorphisms. We denote by $\prob(X)$ the space of all {regular Borel probability measures} on $X$, and equip it with the weak$^*$ topology; this turns $\prob(X)$ into a compact convex space on which $\Gamma$ acts by affine homeomorphisms.

A $\Gamma$-space $X$ is called a \textit{$\Gamma$-boundary} if, for every
$\nu\in\prob(X)$ and every $x\in X$, there is a net $(g_i)$ in $\Gamma$
such that $g_i\nu\to\delta_x$ in the weak$^*$ topology.

Furstenberg~\cite{Furstenberg1973,Furstenberg} showed that for every group $\Gamma$, there is a universal $\Gamma$-boundary $\partial_F\Gamma$ such that every
$\Gamma$-boundary is a continuous $\Gamma$-equivariant image of
$\partial_F\Gamma$. This property determines $\partial_F\Gamma$
up to a $\Gamma$-equivariant homeomorphism. We call $\partial_F\Gamma$ the \textit{Furstenberg boundary} of $\Gamma$.
By a fundamental result of Kalantar-Kennedy~\cite{KK}, $C^*$-simplicity of $\Gamma$ can be characterized via its action $\Gamma\curvearrowright\partial_F\Gamma$. To be precise, $\Gamma$ is $C^*$-simple if and only if its action on $\partial_F\Gamma$ is free.

We will use the following elementary fact that allows us to assume that we can always find a metrizable free $\Gamma$-boundary for a $C^*$-simple group $\Gamma$.
\begin{lemma}
\label{lem:freebndry}
Let $Y$ be a free $\Gamma$-boundary. There exists a free compact
metrizable boundary as a factor of $Y$.
\end{lemma}
\begin{proof}
For each $s\ne e$, the open
sets $\{y\in Y:f(y)\ne f(sy)\}$, with $f\in C(Y)$, cover $Y$.
By compactness, choose a finite family $F_s\subseteq C(Y)$ giving
such a cover. Let $A\subseteq C(Y)$ be the unital $C^*$-algebra
generated by all $\Gamma$-translates of the functions in
$\bigcup_{s\ne e}F_s$. Since $\Gamma$ is countable, $A$ is separable.
Its spectrum $X$ is compact metrizable, and the inclusion
$A\subseteq C(Y)$ induces a continuous equivariant surjection
$p:Y\to X$. If $sx=x$ for some $s\ne e$, choose $y\in p^{-1}(x)$.
Then $p(sy)=p(y)$, so $f(sy)=f(y)$ for every $f\in A$, contradicting
the choice of $F_s$. Thus the action on $X$ is free. Finally, since boundary actions pass to factors, we see that $X$ is a $\Gamma$-boundary.
\end{proof}
\subsection{The Furstenberg URS}
The notion of uniformly recurrent subgroups was introduced by Glasner-Weiss~\cite{glasner2015uniformly} and is intimately connected to the structure of $C_r^*(\Gamma)$ by results of Kennedy~\cite{Kennedy} and Le Boudec-Matte Bon~\cite{LBMB}.
A \textit{uniformly recurrent subgroup}, or a \textit{URS}, of a group
$\Gamma$ is a nonempty closed $\Gamma$-invariant subset of $\Sub(\Gamma)$ which is
minimal for these properties. A URS is called \textit{amenable} if all its members are amenable.

We write $\cA_\Gamma:=\{\Gamma_x : x\in \partial_F\Gamma\}$ for the \emph{Furstenberg URS of $\Gamma$}. By \cite[Theorem~2.16]{LBMB}, for any amenable URS $\cU$ of $\Gamma$, and for every $K\in \cA_\Gamma$ there exists $J\in\cU$ such that $J\le K$. If $E$ is a nonempty closed $\Gamma$-invariant subset of $\Sub^{\mathrm{am}}(\Gamma)$, then it contains an amenable URS by Zorn's lemma, and so for every $K\in \cA_\Gamma$ there exists $J\in E$ such that $J\le K$.

Given $H\le\Gamma$, we regard $\Sub(H)$ as a closed subspace of
$\Sub(\Gamma)$. In particular, $\cA_H$ is a closed subset of $\Sub^{\rm am}(\Gamma)$, and it follows by its universal property that $\cA_{gHg^{-1}}=g\cA_Hg^{-1}$ for $H\le\Gamma$ and $ g\in\Gamma$.

Following \cite{BDL}, for a compact metrizable space $X$, we denote by $\CC(\prob(X))$ the
space of nonempty compact convex subsets of $\prob(X)$, and equip it with the topology defined in \cite[Section~2]{BDL}. This topology is the coarsest topology making the functionals $h_f: \CC(\prob(X))\to \bR$, $\cC\mapsto \max_{\nu\in\cC}\int_Xfd\nu$,
continuous for every $f\in C(X,\bR)$ with $\|f\|_{\infty}\le 1$.
By \cite[Lemma~2.1 and its proof]{BDL}, this space is compact and metrizable.
Indeed, since the above functionals $h_f$ separate compact convex sets by the Hahn--Banach separation theorem, it follows that this topology coincides with the topology of $\CC(\prob(X))$ inherited from $\cK(\prob(X))$.

\Needspace{7\baselineskip}
\begin{lemma}
\label{lem:convmeasurable}
Let $X$ be a compact metric space. Then, the map $${\rm Conv}:\cK(\prob(X))\setminus\{\varnothing\} \to \CC(\prob(X)), ~A\mapsto \overline{{\rm Conv}}^{w^*}(A)$$
is continuous. If $X$ is a $\Gamma$-space, then the map ${\rm Conv}$ is moreover $\G$-equivariant.
\end{lemma}
\begin{proof}
Indeed, since every continuous affine function takes its extreme values on extreme points, for every
$A\in\cK(\prob(X))\setminus\{\varnothing\}$ and $f\in C(X,\bR)$, $$h_f\bigl(\overline{\Conv}^{\,w^*}A\bigr)
 =\max_{\nu\in A}\nu(f).$$
It is straightforward to see that the map $A\mapsto \max_{\nu\in A}\nu(f)$ is continuous on $\cK(\prob(X))\setminus\{\varnothing\}$ for every $f\in C(X,\bR)$. This implies the continuity of the map ${\rm Conv}:\cK(\prob(X))\setminus\{\varnothing\} \to \CC(\prob(X))$.
The last assertion about equivariance is straightforward.
\end{proof}

\section{Proof of the Main result}
In preparation for the proof of the main result, we begin by proving the measurability of certain naturally occurring maps of independent interest.

Let $X$ be a compact metric $\G$-space.
For $H\le \Gamma$, we denote by $\prob_H(X)\subseteq \prob(X)$ the closed convex subset of $H$-invariant measures, which is non-empty if $H$ is amenable.
For $E\subset \Sub(\Gamma)$ we denote $\prob_E(X):=\bigcup_{H\in E} \prob_H(X)$. Observe that if $E$ is a closed subset of $\Sub(\Gamma)$, then $\prob_E(X)$ is a closed subset of $\prob(X)$.
By the universal property recalled in the previous section, $\prob_{\cA_\Gamma}(X)\subseteq\prob_E(X)$ for every nonempty closed $\Gamma$-invariant subset $E\subseteq\Sub^{\mathrm{am}}(\Gamma)$, since every $K$-invariant measure is $J$-invariant whenever $J\le K$.

\begin{lemma}\label{lem:union-measurable}
Let $\Gamma$ be a countable discrete group and let $X$ be a compact
metrizable $\Gamma$-space. The map
$Q:\cK(\Sub^{\mathrm{am}}(\Gamma))\xrightarrow[]{}\cK(\prob(X))$ defined by $Q(E)=\prob_E(X)$
is $\Gamma$-equivariant and Borel measurable.
\end{lemma}

\begin{proof}
For $H\le \Gamma$, we have $\prob_{gHg^{-1}}(X)=\{g\nu: \nu\in \prob_{H}(X)\}$, which in turn implies that $Q$ is $\Gamma$-equivariant.
To see it is Borel measurable, as remarked in the previous section, it is enough to show that the set $\{E\in \cK(\Sub^{\mathrm{am}}(\Gamma)):Q(E)\cap D\ne\varnothing\}$
is closed for every closed set  $D\subseteq\prob(X)$. For this,
let $E_n\to E$ in $\cK(\Sub^{\mathrm{am}}(\Gamma))$, and let $H_n\in E_n$ and $\nu_n\in D\cap\prob_{H_n}(X)$. By passing to a subsequence, we may assume $H_n\to H\in E$ and $\nu_n\to\nu\in D$. This gives $\nu\in\prob_H(X)$, and the claim follows.
Hence, $Q$ is Borel measurable.
\end{proof}
We now turn to the Furstenberg URS map. We first describe a general
construction for compact partially (pre)ordered spaces.

We say a partial preorder $\preccurlyeq$ on a topological space $Y$ is \emph{closed} if the set $\{(y, z) : y\preccurlyeq z\}$ is closed in $Y\times Y$.
In this case, for every $z\in Y$, the set $Y_{\preccurlyeq z}:=\{y\in Y: y\preccurlyeq z\}$ is closed.

For $A\in\cK(Y)$, the set $\Min(A, \preccurlyeq):=\{a\in A : a\preccurlyeq b ~~\forall b\in A\}=A\cap\bigcap_{b\in A}Y_{\preccurlyeq b}$ is closed.
Then, the set $\Max(A, \preccurlyeq):=\{a\in A : b\preccurlyeq a ~~\forall b\in A\} = \Min(A, \preccurlyeq^{op})$ is also closed, where $a\preccurlyeq^{op} b \iff b\preccurlyeq a$ is the opposite preorder on $Y$.

\begin{lemma}\label{lem:ordered-selection}
Let $Y$ be a compact metrizable space equipped with a closed preorder
{$\preccurlyeq$}.
The maps $\Min(\cdot, \preccurlyeq)$ and $\Max(\cdot, \preccurlyeq)$ on $\cK(Y)$ are Borel.
\end{lemma}

\begin{proof}
Let $C\in \cK(Y)$, and let $(A_n)$ be a sequence in $\cK(Y)$ such that $\Min(A_n, \preccurlyeq) \cap C\ne \varnothing$ and $A_n\to A\in \cK(Y)$. Choose a sequence $y_n\in \Min(A_n, \preccurlyeq) \cap C$, and by passing to a subsequence, assume $y_n\to y\in Y$. Since $C$ is closed, $y\in C$, and since $A_n\to A$, we have $y\in A$. Let $a\in A$. Then there exists $a_n\in A_n$ such that $a_n\to a$. Since $y_n \preccurlyeq a_n$, closedness of $\preccurlyeq$ implies $y\preccurlyeq a$. Thus, $y\in \Min(A, \preccurlyeq) \cap C$. This implies that the map $A\mapsto \Min(A, \preccurlyeq)$ on $\cK(Y)$ is Borel measurable. Considering  $\preccurlyeq^{op}$, we also conclude that the map $A\mapsto \Max(A, \preccurlyeq)$ is Borel.
\end{proof}

{Now, we define the partial preorder $\sqsubseteq$ on the space $\cK(\Sub^{\mathrm{am}}(\G))$ by declaring $E\sqsubseteq E'$}
if every $K\in E'$ contains some $J\in E$.
{To see $\sqsubseteq$ is closed, let}
$E_n\to E$, $E'_n\to E'$, and $E_n\sqsubseteq E'_n$. Given
$K\in E'$, choose $K_n\in E'_n$ with $K_n\to K$, and then
$J_n\in E_n$ with $J_n\leq K_n$. Passing to a subsequence, assume
$J_n\to J\in E$. Then $J\leq K$, and we conclude $E\sqsubseteq E'$.
{Recall from the last section that for every $H\le \G$ and every $E\in \cK^H(\Sub^{\mathrm{am}}(H))$, we have $E\sqsubseteq\cA_H$. Thus, $\cA_H$ is a greatest element of $\cK^H(\Sub^{\mathrm{am}}(H))$ for $\sqsubseteq$}, and in fact, using \cite[Proposition~2.14 and Corollary~2.15]{LBMB}, we see
\[
\Max\big(\cK^H(\Sub^{\mathrm{am}}(H)), \sqsubseteq\big) ~\subseteq~ \big\{E\in \cK^H(\Sub^{\mathrm{am}}(H)) : \cA_H\subseteq E\big\}.
\]

{Note that the usual inclusion also defines a closed partial order on $\cK(\Sub^{\mathrm{am}}(\G))$.}

\Needspace{7\baselineskip}
\begin{lemma}\label{lem:Ami}
The map ${\rm Ami}: \Sub(\G)\to \cK(\cK(\Sub^{\mathrm{am}}(\G)))$ defined by $H\mapsto \cK^H(\Sub^{\mathrm{am}}(H))$ is Borel.
\end{lemma}
\begin{proof}
Let $C\in\cK(\cK(\Sub^{\mathrm{am}}(\G)))$, and let
$H_n\to H$ in $\Sub(\G)$ such that
$\cK^{H_n}(\Sub^{\mathrm{am}}(H_n))\cap C\ne\varnothing$.
Choose $E_n\in\cK^{H_n}(\Sub^{\mathrm{am}}(H_n))\cap C$.
By passing to a subsequence, assume $E_n\to E\in C$. Let $L\in E$, and choose $L_n\in E_n$ such that $L_n\to L$.
Since $L_n\le H_n$, we have $L\le H$. Thus
$E\subseteq\Sub^{\mathrm{am}}(H)$.
Moreover, for $h\in H$, we have $h\in H_n$ eventually, so
$hE_nh^{-1}=E_n$ eventually, and hence $hEh^{-1}=E$.
Since the empty set is isolated in the hyperspace, $E$ is nonempty.
Therefore, $E\in\cK^H(\Sub^{\mathrm{am}}(H))\cap C$ and the claim follows.
\end{proof}
\begin{prop}\label{prop:AH-borel}
For every countable discrete group $\Gamma$, the map
$\mathcal F_{\mathrm{URS}}:\Sub(\Gamma)\to
\cK(\Sub^{\mathrm{am}}(\Gamma))$, defined by
$\mathcal F_{\mathrm{URS}}(H)=\cA_H$, is $\Gamma$-equivariant and Borel measurable.
\end{prop}
\begin{proof}
We have
\[
\mathcal F_{\mathrm{URS}}(\cdot) = \Min\Big(\Max\big({\rm Ami}(\cdot),\sqsubseteq\big), \subseteq\Big)
\]
and therefore it is Borel by Lemma~\ref{lem:ordered-selection} and Lemma~\ref{lem:Ami}.

The equivariance of the map follows from the fact that $g\cA_Hg^{-1}=\cA_{gHg^{-1}}$ for every $g\in \G$ and $H\le \G$.
\end{proof}
Combining all the above measurability results, we finish the proof of the main theorem.
\begin{proof}[Proof of Theorem~\ref{thm:maintheorem}]
Let $\Gamma$ be a countable discrete $C^*$-simple group and let $\mu$
be an IRS of $\Gamma$. Its action on $\partial_F\Gamma$ is free
by \cite{KK}. Using Lemma~\ref{lem:freebndry}, we obtain a free compact
metrizable $\Gamma$-boundary $X$.

By Lemma~\ref{lem:convmeasurable}, Lemma~\ref{lem:union-measurable}, and Proposition~\ref{prop:AH-borel}, the map $\Phi = {\rm Conv}\circ Q \circ \mathcal{F}_{\mathrm{URS}}: \Sub(\G)\to \CC(\prob(X))$ is $\G$-equivariant and Borel measurable.
Consequently, $\Phi_*\mu$ is a $\Gamma$-invariant Borel
probability measure on $\CC(\prob(X))$. Since $X$ is a boundary,
$\prob(X)$ is an irreducible affine $\Gamma$-space. Hence it follows from
\cite[Lemma~2.3]{BDL} that $\Phi_*\mu=\delta_{\prob(X)}$.
Therefore, there is a measurable $\mu$-conull set $\Omega\subseteq\Sub(\Gamma)$ such that $\overline{\Conv}^{\,w^*}\prob_{\cA_H}(X)=\prob(X)$ for every $H\in\Omega$.
Fix $H\in\Omega$ and $x\in X$.
By Milman's converse to the Krein--Milman theorem, $\delta_x\in\prob_{\cA_H}(X)$.
But since the action of $\G$ on $X$ is free, it follows that $\{e\}\in\cA_H$, and so $\cA_H=\{\{e\}\}$.
Hence, $H$ is $C^*$-simple for $\mu$-almost every $H\in \Sub(\G)$.
\end{proof}
\section{Acknowledgements}
\noindent
The second author would like to thank Yair Hartman for many valuable discussions around this problem.
\\
This work was completed when the second named author was visiting the Institute of Mathematics of the Polish Academy of Sciences, Warsaw. We thank the Institute for its hospitality.
\\
For the purpose of open access, the authors have applied a CC-BY license to any author accepted manuscript arising from this submission.

\bibliographystyle{amsalpha}
\bibliography{name}
\end{document}